\documentclass[3p]{elsarticle}

\usepackage{hyperref, amssymb, amsmath, amsthm}
\usepackage[font=scriptsize]{caption}%

\newtheorem{theorem}{Theorem}[section]
\newtheorem{lemma}[theorem]{Lemma}
\newtheorem{proposition}[theorem]{Proposition}
\newtheorem{corollary}[theorem]{Corollary}

\theoremstyle{definition}
\newtheorem{definition}[theorem]{Definition}
\newtheorem{example}[theorem]{Example}

\theoremstyle{remark}
\newtheorem{remark}[theorem]{Remark}

\journal{Linear Algebra and its Applications}

\begin{document}

\begin{frontmatter}

\title{Generalized eigenvalue problems for meet and join matrices on semilattices}

\author{Pauliina Ilmonen}
\ead{pauliina.ilmonen@aalto.fi}
\author{Vesa Kaarnioja\corref{cor1}\fnref{support}}
\ead{vesa.kaarnioja@aalto.fi}
\cortext[cor1]{Corresponding author}
\address{Aalto University School of Science, Department of Mathematics and Systems Analysis, P.O. Box 11100, FI-00076 Aalto, Finland}
\fntext[support]{The author has been supported by the Academy of Finland (decision 267789).}




\begin{abstract}
We study generalized eigenvalue problems for meet and join matrices with respect to incidence functions on semilattices. We provide new bounds for generalized eigenvalues of meet matrices with respect to join matrices under very general assumptions. The applied methodology is flexible, and it is shown in the case of GCD and LCM matrices that even sharper bounds can be obtained by applying the known properties of the divisor lattice. These results can also be easily modified for the dual problem of eigenvalues of join matrices with respect to meet matrices, which we briefly consider as well. We investigate the effectiveness of the obtained bounds for select examples involving number-theoretical lattices.
\end{abstract}


\end{frontmatter}


\section{Introduction}
The mathematical literature concerning the properties of GCD and related matrices can be traced back to Smith~\cite{Smith}, who in 1876 studied the determinant of the $n\times n$ matrix having the greatest common divisor of $i$ and $j$ as its $ij$ entry. This matrix is affectionately called the Smith matrix by later authors. The generalizations of classical GCD and LCM matrices are known as meet and join matrices, respectively, and they arise in lattice theory, where they are used to represent the interrelations between distinct lattice elements. The study of eigenvalues and determinant of meet and join matrices relies heavily on factorizations of these matrices. Factorizations can be constructed by separating the partial order relation endowed on the lattice from the contribution of the incidence function on the lattice elements.

The eigenvalues and determinant of GCD and LCM matrices have been a source of interest ever since the pioneering work of Smith. Wintner~\cite{Wi}, and Lindqvist and Seip~\cite{Li} considered the eigenvalues of the $n\times n$ matrix having
\[
\frac{{\rm gcd}(i,j)^\alpha}{{\rm lcm}(i,j)^\alpha}
\]
as its $ij$ entry. Beslin and Ligh~\cite{BeLi} proved that GCD matrices are positive definite and this result was later shown by Bourque and Ligh~\cite{BoLI} to hold for all power GCD matrices for any $\alpha>0$. Ovall~\cite{Ov} continued this mode of research by considering the positive definiteness of matrices related to GCD matrices. Balatoni~\cite{Fe} estimated the smallest and largest eigenvalue of the Smith matrix. Hong and Loewy~\cite{AR} investigated the asymptotic behavior of eigenvalues of power GCD matrices. Bhatia~\cite{Bha} considered GCD matrices as an example of infinitely divisible matrices.

Ilmonen, Haukkanen, and Merikoski~\cite{i1} instigated the study eigenvalues of meet and join matrices associated with incidence functions. The eigenvalues of meet and join matrices have since been studied thoroughly in recent literature: Mattila and Haukkanen~\cite{Mattila4} studied positive definiteness and eigenvalues of meet and join matrices, and Mattila, Haukkanen, and M\"{a}ntysalo~\cite{Mattila5} considered the singularity of LCM-type matrices.

Generalized eigenvalue problems associated with meet and join matrices have not been studied in the literature. In this work, we conduct a first study of generalized eigenvalues of meet and join matrices with respect to incidence functions on semilattices and we provide new bounds for both minimal and dominant eigenvalues. Our results provide a general framework within which it is possible to produce bounds for generalized eigenvalues of very complicated lattice-theoretical constructions, and we are confident that the methodology used in this paper can be successfully employed to other problems of interest.

This paper is organized as follows. In Section~\ref{notations}, we review the essential notations surrounding meet and join matrices and incidence functions defined on semilattices. The preliminary decomposition theory governing meet and join matrices is also reviewed under this section. In Section~\ref{gcdlcmsec}, we introduce the generalized eigenvalue problem for meet matrices with respect to join matrices and proceed to prove new bounds for the dominant and minimal eigenvalues. We provide global bounds, which describe the behavior of the generalized spectrum of meet matrices with respect to join matrices. Moreover, we also consider so-called local bounds that employ the corresponding lattice structure in more detail. When the structure of the lattice is known --- as is the case for the divisor lattice, for example --- this produces improved bounds. We also briefly present the corresponding results for the dual eigenvalue problem in Section~\ref{lcmgcdsec}. We investigate the obtained bounds for select examples in Section~\ref{experiments} and end with some conclusions and future prospects in Section~\ref{conclusions}.

\section{Notations and preliminaries}\label{notations}
\subsection{Meet and join matrices}\label{meetandjoin}
Let $(P,\preceq)$ be a nonempty poset and $f$ a complex-valued function on $P$. The poset is called \emph{locally finite} if the set
\[
\{z\in P\mid x\preceq z\preceq y\}
\]
is finite for all $x,y\in P$. If the greatest lower bound of $x,y\in P$ exists, it is called the \emph{meet} of $x$ and $y$  and is denoted by $x\wedge y$. If the least upper bound of $x,y\in P$ exists, it is called the \emph{join} of $x$ and $y$ and is denoted by $x\vee y$. If $x\wedge y\in P$ exists for all $x,y\in P$, then $(P,\preceq,\wedge)$ is called a \emph{meet semilattice}, and if $x\vee y\in P$ exists for all $x,y\in P$, then $(P,\preceq,\vee)$ is called a \emph{join semilattice}. If the poset $(P,\preceq,\wedge,\vee)$ is both a meet semilattice and a join semilattice, then it is called a \emph{lattice}.

Let $(P,\preceq)$ be a poset and suppose that $S=\{x_1,x_2,\ldots,x_n\}$ is a finite subset of $P$ such that $x_i\preceq x_j$ only if $i\leq j$. If $(P,\preceq,\wedge)$ is a meet semilattice, then the $n\times n$ matrix $(S)_f$ defined by setting
\[
((S)_f)_{i,j}=f(x_i\wedge x_j),\quad i,j\in\{1,2,\ldots,n\},
\]
is called the \emph{meet matrix} on $S$ with respect to $f$. If $(P,\preceq,\vee)$ is a join semilattice, then the $n\times n$ matrix $[S]_f$ defined by setting
\[
([S]_f)_{i,j}=f(x_i\vee x_j),\quad i,j\in\{1,2,\ldots,n\},
\]
is called the \emph{join matrix} on $S$ with respect to $f$.

A complex-valued function $f$ on $P\times P$ such that $f(x,y)=0$, whenever $x\not\preceq y$, is called an \emph{incidence function} of $P$. If $f$ and $g$ are incidence functions of $P$, then their sum $f+g$ is defined by
\[
(f+g)(x,y)=f(x,y)+g(x,y)\quad\text{for all }x,y\in P,
\]
their product $fg$ is given by
\[
(fg)(x,y)=f(x,y)g(x,y)\quad\text{for all }x,y\in P,
\]
and their convolution is defined by
\[
(f*g)(x,y)=\sum_{x\preceq z\preceq y}f(x,z)g(z,y)\quad\text{for all }x,y\in P.
\]
The sum, product, and convolution of incidence functions of $P$ are incidence functions of $P$ as well.

The incidence function $\delta$ defined by
\[
\delta(x,y)=\begin{cases}1&\text{if }x=y,\\ 0&\text{otherwise}\end{cases}
\]
is unity under the convolution. The incidence function $\zeta$ of $P$ is defined by
\[
\zeta(x,y)=\begin{cases}1&\text{if }x\preceq y,\\ 0&\text{otherwise.}\end{cases}
\]
The inverse of $\zeta$ under the convolution is the M\"{o}bius function $\mu$ on $P$ defined inductively by
\begin{align*}
&\mu(x,x)=1,\\
&\mu(x,y)=-\sum_{x\preceq z\prec y}\mu(x,z)\quad \text{if }x\prec y
\end{align*}
for all $x,y\in P$. 
\subsection{The divisor lattice}
An important special case of the general framework of Subsection~\ref{meetandjoin} is the divisor lattice $(P,\preceq)=(\mathbb{Z_+},|)$, where $|$ denotes the ordinary divisibility relation of positive integers. The greatest lower bound of $x,y\in\mathbb{Z_+}$ is their \emph{greatest common divisor} (gcd)
\[
x\wedge y=(x,y)
\]
and their least upper bound is their \emph{least common multiple} (lcm)
\[
x\vee y=[x,y].
\]
The quadruplet $(\mathbb{Z_+},|,{\rm gcd},{\rm lcm})$ is a locally finite lattice possessing the smallest element $1\in\mathbb{Z_+}$.

A complex-valued function $f$ on $\mathbb{Z_+}$ is called an arithmetical function. The Dirichlet convolution of arithmetical functions $f$ and $g$ is defined by
\[
(f*_Dg)(x)=\sum_{y|x}f(y)g\bigg(\frac{x}{y}\bigg),\quad x\in\mathbb{Z_+}.
\]
The arithmetical function $\delta$ defined by
\[
\delta(x)=\begin{cases}1&\text{if }x=1,\\ 0&\text{otherwise}.\end{cases}
\]
is unity under the Dirichlet convolution. The function $\zeta$ is defined by
\[
\zeta(x)=1\quad\text{for all }x\in\mathbb{Z_+}
\]
and the inverse of $\zeta$ under the Dirichlet convolution is the arithmetical M\"{o}bius function $\mu$ given by
\[
\mu(x)=\begin{cases}1&\text{if }x=1,\\ (-1)^n&\text{if $x$ is the product of $n$ distinct prime numbers},\\ 0&\text{otherwise.}\end{cases}
\]

The respective incidence functions $\delta$, $\zeta$, and $\mu$ of the lattice $(\mathbb{Z_+},|)$ are obtained by identifying
\[
\delta(x,y)=\delta\bigg(\frac{y}{x}\bigg),\ \zeta(x,y)=\zeta\bigg(\frac{y}{x}\bigg),\ \text{and}\ \mu(x,y)=\mu\bigg(\frac{y}{x}\bigg)\quad\text{for }x|y,\ y\in\mathbb{Z_+},
\]
which justifies the dual use of notations. In the sequel, the intended meaning of the functions $\delta$, $\zeta$, and $\mu$ is clear from context. For further material on arithmetical functions and incidence algebras, we refer the interested reader to~\cite{Ai,Mc,St}.

\subsection{Structure theory of meet and join matrices}
The properties of meet matrices have been researched extensively by many authors and they are known to possess a variety of desirable qualities under very general assumptions. In this section, we state some of the known properties of meet matrices on locally finite meet semilattices.

Let $(P,\preceq,\wedge,\hat 0)$ be a locally finite meet semilattice with the least element $\hat 0$, i.e., $\hat 0\preceq x$ for all $x\in P$, and suppose that $S=\{x_1,x_2,\ldots,x_n\}$ is a finite subset of $P$ such that $x_i\preceq x_j$ only if $i\leq j$. The set $S$ is said to be \emph{lower closed} if $y\in S$ whenever $x\in S$, $y\in P$ with $y\preceq x$, and $S$ is said to be \emph{meet closed} if $x\wedge y\in S$ for all $x,y\in S$. Let $f$ be a complex-valued function on $P$. We associate $f$ with a restricted incidence function $f_d$ of $(P,\preceq,\wedge,\hat 0)$ defined by the formula
\[
f_d(\hat 0,z)=f(z),\quad z\in P.
\]

The following factorization exists for meet matrices on meet closed sets.
\begin{proposition}[cf.~{\cite[Theorem~12]{Ra}}]\label{EDET}
Let $S$ be meet closed and define the $n\times n$ matrices $E$ and $D={\rm diag}(d_1,\ldots,d_n)$ by setting
\begin{align}
&E_{i,j}=\begin{cases}1&\text{if }x_j\preceq x_i,\\ 0&\text{otherwise},\end{cases}\quad i,j\in\{1,2,\ldots,n\},\label{Emat}\\
&d_i=\sum_{\substack{z\preceq x_i\\ z\not\preceq x_j,\ j<i}}(f_d*\mu)(\hat 0,z),\quad i\in\{1,2,\ldots,n\}.\label{Dmat}
\end{align}
Then $(S)_f=EDE^{\textup{T}}$.
\end{proposition}
\begin{remark}\label{pdremark}
It is evident from Proposition~\ref{EDET} that $(S)_f$ is positive semidefinite for meet closed $S$ if and only if
\[
d_i=\sum_{\substack{z\preceq x_i\\ z\not\!\preceq x_j,\ j<i}}(f_d*\mu)(\hat 0,z)\geq 0\quad \text{for all }i\in\{1,2,\ldots,n\}
\]
and $(S)_f$ is positive definite if and only if $d_i>0$ for all $i\in\{1,2,\ldots,n\}$.
\end{remark}
On lower closed sets, the outer sum in the expression~\eqref{Dmat} vanishes.
\begin{proposition}[cf.~{\cite[Example~1]{Ha}}]
Let $S$ be lower closed. Then
\[
f(x_i)=\sum_{\substack{z\preceq x_i\\ z\not\preceq x_j,\ j<i}}f(z),\quad x_i\in S.
\]
\end{proposition}
For example, in the lattice $(P,\preceq)=(\mathbb{Z_+},|)$, the restricted incidence function $f(x)=f_d(\hat 0,x)=x^\alpha$, $\alpha\in\mathbb{R}$, yields $(f_d*\mu)(x)=J_{\alpha}(x)$, where $J_\alpha$ denotes Jordan's totient function. In the special case $\alpha=1$, this convolution is equal to Euler's totient function $J_1(x)=\phi(x)$.

It has been shown by Korkee and Haukkanen~\cite{Is2} that join matrices inherit the structure theory governing meet matrices under the following condition imposed on the restricted incidence function.
\begin{definition}\label{semimultipl}
Let $(P,\preceq,\wedge,\vee)$ be a poset and $f$ a complex-valued function defined on $P$. The function $f$ is \emph{semimultiplicative} if
\[
f(x\wedge y)f(x\vee y)=f(x)f(y)\quad\text{for all }x,y\in P.
\]
\end{definition}
Join matrices possess the following decomposition.
\begin{proposition}[cf.~{\cite[Lemma~5.1]{Is2}}]\label{RELETR}
Let $f$ be a semimultiplicative function on $P$ such that $f(x)\neq 0$ for all $x\in P$ and let $R={\rm diag}(f(x_1),\ldots,f(x_n))$. Then
\[
[S]_f=R(S)_{1/f}R.
\]
\end{proposition}

Proposition~\ref{RELETR} may be understood as the matrix analogue of the semimultiplicative property of Definition~\ref{semimultipl}.

The estimation of the norms of the factors comprising meet and join matrices plays a crucial role in the following section. To this end, we observe that the matrix $E$ in~\eqref{Emat} belongs to the matrix algebra $K(n)$ of all $n\times n$ lower triangular $0,1$ matrices such that each main diagonal entry is equal to $1$. Clearly every matrix $X\in K(n)$ is real and nonsingular and thus $XX^\textup{T}$ is positive definite. We now define the positive constants $c_n$~\cite{AR} and $C_n$~\cite{i1} depending only on $n$ such that
\[
c_n=\min\{\lambda\mid X\in K(n),\ \lambda\ \text{is the smallest eigenvalue of }XX^\textup{T}\}
\]
and
\[
C_n=\max\{\lambda\mid X\in K(n),\ \lambda\ \text{is the largest eigenvalue of }XX^\textup{T}\}.
\]
In the following sections, we use the constants $c_n$ and $C_n$ to construct bounds for generalized eigenvalues of meet and join matrices.

Calculating the values of $c_n$ and $C_n$ directly from their respective definitions is intractable for even moderate values of $n$ since the number of elements in $K(n)$ grows exponentially as $\#K(n)=2^{n(n-1)/2}$. However, it has been shown recently that the representative matrices in $K(n)$ corresponding to the constants $c_n$ and $C_n$ are known a priori. In the paper~\cite{i1}, it has been demonstrated that the symmetric $n\times n$ matrix $M_n$ satisfying
\[
C_n=\|M_n\|,
\]
where $\|\cdot\|$ denotes the spectral norm, is given by
\[
(M_n)_{i,j}=\min\{i,j\},\quad i,j\in\{1,2,\ldots,n\}.
\]

In a recent work, Alt{\i}n{\i}\c{s}{\i}k et al.~\cite{ihk} have conversely proven that the matrix $Y_n\in K(n)$ defined by
\[
(Y_n)_{i,j}=\begin{cases}1,&\text{if }i=j,\\ \frac{1-(-1)^{i+j}}{2},&\text{if }i>j,\\ 0&\text{otherwise}\end{cases}
\]
satisfies $c_n=\|Y_nY_n^\textup{T}\|$. These identities can be used to efficiently compute of the coefficients $c_n$ and $C_n$ using readily available mathematical software such as MATLAB or Mathematica.

\section{Eigenvalues of meet matrices with respect to join matrices}\label{gcdlcmsec}
Let $(P,\preceq,\wedge,\vee,\hat 0)$ be a poset, $f$ a complex-valued, semimultiplicative function defined on $P$, and suppose that $S=\{x_1,x_2,\ldots,x_n\}$ is a finite subset of $P$ such that $x_i\preceq x_j$ only if $i\leq j$. We consider the generalized eigenvalue problem of finding $(\lambda,x)\in\mathbb{C}\times(\mathbb{C}^n\setminus\{0\})$ such that
\begin{equation}
(S)_fx=\lambda[S]_fx,\label{geneig}
\end{equation}
where $\lambda$ is called the $[S]_f$-eigenvalue of $(S)_f$ and $x$ the corresponding eigenvector.

In order to be able to analyze the spectrum of the system~\eqref{geneig}, we first introduce the following lemma showing that the spectrum is equivalent to the eigenvalues of another matrix.
\begin{lemma}\label{gcdlcmlemma}
Let $(P,\preceq,\wedge,\vee,\hat 0)$ be a poset and $f$ a complex-valued, semimultiplicative function defined on $P$ such that $f(x)\neq 0$ for all $x\in P$ and define $g(x)=1/f(x)$. Let $S=\{x_1,x_2,\ldots,x_n\}$ be a finite, meet closed subset of $P$ such that $x_i\preceq x_j$ only if $i\leq j$. Let
\begin{equation}
l_i=\sum_{\substack{z\preceq x_i\\ z\not\preceq x_j,\ j<i}}(g_d*\mu)(\hat 0,z)\neq 0\quad\text{for all }i\in\{1,2,\ldots,n\}.\label{invassumpt}
\end{equation}
Then the $[S]_f$-eigenvalues $\lambda\in\mathbb{C}$ solving
\[
(S)_fx=\lambda[S]_fx
\]
for some $x\in\mathbb{C}^n\setminus\{0\}$ are precisely the eigenvalues of the system
\[
L^{-1}Ay=\lambda y,
\]
where $y\in\mathbb{C}^n\setminus\{0\}$, $L={\rm diag}(l_1,\ldots,l_n)$, $A=(RE)^{-1}(S)_f(E^\textup{T}R)^{-1}$, and $R={\rm diag}(f(x_1),\ldots,f(x_n))$ and $E$ is the $n\times n$ matrix defined in~\eqref{Emat}.
\end{lemma}
\begin{proof}
Let $(\lambda,x)\in\mathbb{C}\times\mathbb{C}^n\setminus\{0\}$ be an eigenpair solution of the system
\[
(S)_fx=\lambda[S]_fx.
\]
By Propositions~\ref{EDET} and~\ref{RELETR}, these matrices have the decompositions $(S)_f=EDE^\textup{T}$ and $[S]_f=RELE^\textup{T}R$, where $E$ is given by~\eqref{Emat}, $R={\rm diag}(f(x_1),\ldots,f(x_n))$, and $D={\rm diag}(d_1,\ldots,d_n)$ and $L={\rm diag}(l_1,\ldots,l_n)$ are diagonal matrices with elements
\[
d_i=\sum_{\substack{z\preceq x_i\\ z\not\preceq x_j,\ j<i}}(f_d*\mu)(\hat 0,z)\ \text{and}\ l_i=\sum_{\substack{z\preceq x_i\\ z\not\preceq x_j,\ j<i}}(g_d*\mu)(\hat 0,z),\quad i\in\{1,2,\ldots,n\}.
\]
The assumption~\eqref{invassumpt} guarantees that $[S]_f$ is invertible, which yields
\begin{align*}
&(S)_fx=\lambda [S]_fx\\
\Leftrightarrow\quad &EDE^\textup{T}x=\lambda RELE^\textup{T}Rx\\
\Leftrightarrow\quad &L^{-1}(RE)^{-1}EDE^\textup{T}(E^\textup{T}R)^{-1}y=\lambda y,
\end{align*}
where $y=E^\textup{T}Rx$.
\end{proof}

If the meet matrix is positive semidefinite, then the $[S]_f$-eigenvalues of $(S)_f$ are real and its $[S]_f$-inertia (i.e., the signs of its $[S]_f$-eigenvalues) is known.
\begin{corollary}\label{gcdlcmsigns}
In addition to the assumptions of Lemma~\ref{gcdlcmlemma}, let
\[
d_i=\sum_{\substack{z\preceq x_i\\ z\not\preceq x_j,\ j<i}}(f_d*\mu)(\hat 0,z)\geq 0\quad\text{for all }i\in\{1,2,\ldots,n\}.
\]
Then the $[S]_f$-eigenvalues of $(S)_f$ are real and they can be ordered $(\lambda_i)_{i=1}^n$ (counting multiplicities) such that
\[
{\rm sign}\,\lambda_i={\rm sign}\sum_{\substack{z\preceq x_i\\ z\not\preceq x_j,\ j<i}}(g_d*\mu)(\hat 0,z)\quad\text{for all }i\in\{1,2,\ldots,n\}.
\]
\end{corollary}
\begin{proof}
The diagonal matrix $D={\rm diag}(d_1,\ldots,d_n)$ is positive semidefinite and it follows immediately that the matrices $(S)_f=EDE^\textup{T}$ and $A=(RE)^{-1}(S)_f(E^\textup{T}R)^{-1}$ are positive semidefinite as well. In particular, the principal square root $A^{1/2}$ exists allowing us to carry out the similarity transformation
\[
L^{-1}A\sim A^{1/2}L^{-1}A^{1/2},
\]
where the spectrum of the latter matrix coincides with the $[S]_f$-eigenvalues of $(S)_f$.

The matrix $A^{1/2}L^{-1}A^{1/2}$ is symmetric, which shows that the $[S]_f$-eigenvalues of $(S)_f$ are real. On the other hand, the matrix $A^{1/2}L^{-1}A^{1/2}$ is a congruence transformation of $L^{-1}$, meaning that the inertia is invariant between the two by~\cite[Theorem~8.1.12]{GolubVanLoan}.
\end{proof}
\subsection{Global bounds for the $[S]_f$-eigenvalues of $(S)_f$}
In this section, we present, under nonsingularity of $[S]_f$, upper and lower bounds for the $[S]_f$-eigenvalues of $(S)_f$ defined on a finite,  meet closed set $S$ endowed with a nonvanishing, complex-valued semimultiplicative incidence function $f$. By a global bound, we mean a uniform bound that holds for any meet and join matrices which subscribe to these qualities.
\begin{theorem}\label{meetandjoinbound}
Let $(P,\preceq,\wedge,\vee,\hat 0)$ be a poset, $f$ a semimultiplicative function on $P$ such that $f(x)\neq 0$ for all $x\in P$, and let $S=\{x_1,x_2,\ldots,x_n\}$ be a finite, meet closed subset of $P$ such that $x_i\preceq x_j$ only if $i\leq j$. Define $g(x)=1/f(x)$ for $x\in P$ and let
\[
\sum_{\substack{z\preceq x_i\\ z\not\preceq x_j,\ j<i}}(g_d*\mu)(\hat 0,z)\neq 0\quad\text{for all }i\in\{1,2,\ldots,n\}.
\]
Then the $[S]_f$-eigenvalues $\lambda\in\mathbb{C}$ of $(S)_f$ are bounded from above by
\[
|\lambda|\leq\frac{M_fm_g^{-1}C_nc_n^{-1}}{\min_{1\leq i\leq n}|f(x_i)|^2},
\]
where
\[
M_f=\max_{1\leq i\leq n}\bigg|\sum_{\substack{z\preceq x_i\\ z\not\preceq x_j,\ j<i}}(f_d*\mu)(\hat 0,z)\bigg|\ \text{and}\  m_g=\min_{1\leq i\leq n}\bigg|\sum_{\substack{z\preceq x_i\\ z\not\preceq x_j,\ j<i}}(g_d*\mu)(\hat 0,z)\bigg|.
\]
If in addition it holds that
\begin{equation}
\sum_{\substack{z\preceq x_i\\ z\not\preceq x_j,\ j<i}}(f_d*\mu)(\hat 0,z)\neq 0\quad\text{for all }i\in\{1,2,\ldots,n\},\label{gcdlcminv}
\end{equation}
then the $[S]_f$-eigenvalues $\lambda$ of $(S)_f$ are bounded from below by
\[
|\lambda|\geq\frac{m_fM_g^{-1}C_n^{-1}c_n}{\max_{1\leq i\leq n}|f(x_i)|^2},
\]
where
\[
m_f=\min_{1\leq i\leq n}\bigg|\sum_{\substack{z\preceq x_i\\ z\not\preceq x_j,\ j<i}}(f_d*\mu)(\hat 0,z)\bigg|\ \text{and}\  M_g=\max_{1\leq i\leq n}\bigg|\sum_{\substack{z\preceq x_i\\ z\not\preceq x_j,\ j<i}}(g_d*\mu)(\hat 0,z)\bigg|.
\]
\end{theorem}
\begin{proof}
By Propositions~\ref{EDET} and~\ref{RELETR}, we have the decompositions $(S)_f=EDE^\textup{T}$ and $[S]_f=RELE^\textup{T}R$ and it suffices to inspect the spectral radius of $L^{-1}A$, where $A=(RE)^{-1}EDE^\textup{T}(E^\textup{T}R)^{-1}$. We use the spectral norm which we denote by $\|\cdot\|$. Since
\[
\|MM^\textup{T}\|=\|M\|\cdot\|M^\textup{T}\|=\|M\|^2
\]
for any square matrix $M$, the spectral radius of $L^{-1}A$ is bounded by
\begin{align*}
\rho(L^{-1}A)&\leq\|L^{-1}\|\cdot\|A\|=\|L^{-1}\|\cdot\|(RE)^{-1}EDE^\textup{T}(E^\textup{T}R)^{-1}\|\\
&\leq \|L^{-1}\|\cdot \|E^{-1}\|\cdot\|R^{-1}\|\cdot\|E\|\cdot\|D\|\cdot\|E^\textup{T}\|\cdot\|R^{-1}\|\cdot\|E^{-\textup{T}}\|\\
&=(\|R^{-1}\|)^2\cdot \|L^{-1}\|\cdot \|D\|\cdot\|EE^\textup{T}\|\cdot\|(EE^\textup{T})^{-1}\|.
\end{align*}
For the diagonal matrices $R^{-1}$, $L^{-1}$, and $D$, we obtain
\begin{align*}
&\|R^{-1}\|=\left\|{\rm diag}\left(\frac{1}{f(x_1)},\ldots,\frac{1}{f(x_n)}\right)\right\|=\frac{1}{\min_{1\leq i\leq n}|f(x_i)|},\\
&\|L^{-1}\|=\left\|{\rm diag}\left(\frac{1}{l_1},\ldots,\frac{1}{l_n}\right)\right\|=m_g^{-1},\\
&\|D\|=\left\|{\rm diag}(d_1,\ldots,d_n)\right\|=M_f.
\end{align*}

The matrix $E$ belongs to the set $K(n)$ defined in Section~\ref{notations} and hence
\[
\|(EE^\textup{T})^{-1}\|\leq\frac{1}{c_n}
\]
and
\[
\|EE^\textup{T}\|\leq C_n.
\]
Thus
\[
|\lambda|\leq\frac{M_fm_g^{-1}C_nc_n^{-1}}{\min_{1\leq i\leq n}|f(x_i)|^2}.
\]

To obtain the lower bound under the assumption~\eqref{gcdlcminv}, we first inspect the spectral radius of $(L^{-1}A)^{-1}$. To this end, we compute
\begin{align*}
\rho((L^{-1}A)^{-1})&\leq \|A^{-1}\|\cdot\|L\|=\|E^\textup{T}RE^{-\textup{T}}D^{-1}E^{-1}RE\|\cdot\|L\|\\
&\leq \|E^\textup{T}\|\cdot\|R\|\cdot\|E^{-\textup{T}}\|\cdot\|D^{-1}\|\cdot\|E^{-1}\|\cdot\|R\|\cdot\|E\|\cdot\|L\|\\
&=\|R\|^2\cdot\|L\|\cdot\|D^{-1}\|\cdot\|EE^\textup{T}\|\cdot\|(EE^\textup{T})^{-1}\|\\
&\leq \max_{1\leq i\leq n}|f(x_i)|^2\cdot M_gm_f^{-1}C_nc_n^{-1},
\end{align*} 
where similar argumentation is used to obtain the constants as before. For the $[S]_f$-eigenvalues $\lambda\in\mathbb{C}\setminus\{0\}$ of $(S)_f$, this implies that
\[
\frac{1}{|\lambda|}\leq \max_{1\leq i\leq n}|f(x_i)|^2\cdot M_gm_f^{-1}C_nc_n^{-1},
\]
and thus
\[
|\lambda|\geq\frac{m_fM_g^{-1}C_n^{-1}c_n}{\max_{1\leq i\leq n}|f(x_i)|^2}.
\]
This concludes the proof.
\end{proof}
By its nature, a global bound provides a fairly pessimistic estimate for both the dominant and the minimal eigenvalue. In fact, the growth rate of $C_nc_n^{-1}$ is exponential as $n\to\infty$. However, we remark that the obtained upper bound is very similar to the global bounds obtained for the regular eigenvalues of meet matrices in~\cite{i1}.
\subsection{Local bounds for the $[S]_f$-eigenvalues of $(S)_f$}

In this section, we proceed to derive so-called local bounds for the $[S]_f$-eigenvalues of $(S)_f$. The choice of terminology stems from the fact that our local bounds utilize incidental information that is immediately discernible from the lattice's structure. In other words, the partial order relation embedded to a meet semilattice $(P,\preceq,\wedge,\vee)$ uniquely defines the associated M\"{o}bius function $\mu$ of $P$, which can be used to derive improved bounds for generalized eigenvalues by utilizing the properties inherent to the lattice of interest.

We begin by introducing a convenient lemma that characterizes the elements of the auxiliary matrix described in Lemma~\ref{gcdlcmlemma}.
\begin{lemma}\label{gcdlcmelem}In addition to the assumptions of Lemma~\ref{gcdlcmlemma}, let $S$ be lower closed. Then it holds that
\[
(L^{-1}A)_{i,j}=l_i^{-1}\sum_{\alpha,\beta=1}^n\frac{\mu(x_\alpha,x_i)\mu(x_\beta,x_j)}{f(x_\alpha\vee x_\beta)}\quad\text{for }i,j\in\{1,2,\ldots,n\}.
\]\end{lemma}
\begin{proof}
Let $B=L^{-1}(RE)^{-1}EDE^\textup{T}(E^\textup{T}R)^{-1}$, where the factors are defined as in Lemma~\ref{gcdlcmlemma}. Writing open the matrix products and utilizing the fact that $L$, $R$, and $D$ are diagonal matrices yields
\begin{align*}
B_{i,j}&=l_i^{-1}\sum_{\alpha,\alpha',\beta,\beta',\gamma,\gamma'=1}^nE_{i,\alpha}^{-1}R_{\alpha,\alpha'}^{-1}E_{\alpha',\beta}D_{\beta,\beta'}E_{\gamma,\beta'}R_{\gamma,\gamma'}^{-1}E_{j,\gamma'}^{-1}\\
&=l_i^{-1}\sum_{\alpha,\beta,\gamma=1}^nE_{i,\alpha}^{-1}R_{\alpha,\alpha}^{-1}E_{\alpha,\beta}D_{\beta,\beta}E_{\gamma,\beta}R_{\gamma,\gamma}^{-1}E_{j,\gamma}^{-1}\\
&=l_i^{-1}\sum_{\alpha,\gamma=1}^nE_{i,\alpha}^{-1}E_{j,\gamma}^{-1}R_{\alpha,\alpha}^{-1}R_{\gamma,\gamma}^{-1}\sum_{\beta=1}^nD_{\beta,\beta}E_{\alpha,\beta}E_{\gamma,\beta}
\end{align*}
for $i,j\in\{1,2,\ldots,n\}$. It follows from Proposition~\ref{EDET} that
\[
f(x_{\alpha}\wedge x_{\gamma})=(EDE^\textup{T})_{\alpha,\gamma}=\sum_{\beta=1}^nD_{\beta,\beta}E_{\alpha,\beta}E_{\gamma,\beta}
\]
and from the definitions of $E$ of $R$ we have that
\[
E_{i,j}^{-1}=\mu(x_j,x_i)\ \text{and}\  R_{i,i}=f(x_i)^{-1}\quad\text{for }i,j\in\{1,2,\ldots,n\}.
\]
Thus
\begin{align*}
B_{i,j}&=l_i^{-1}\sum_{\alpha,\gamma=1}^n\mu(x_\alpha,x_i)\mu(x_\gamma,x_j)\frac{f(x_\alpha\wedge x_\gamma)}{f(x_\alpha)f(x_\gamma)}=l_i^{-1}\sum_{\alpha,\gamma=1}^n\frac{\mu(x_\alpha,x_i)\mu(x_\gamma,x_j)}{f(x_\alpha\vee x_\gamma)},
\end{align*}
where the semimultiplicative property of $f$ was utilized on the final equality.
\end{proof}

The local bounds for the $[S]_f$-eigenvalues of $(S)_f$ are displayed in the following.
\begin{theorem}\label{localbound1}
Let $(P,\preceq,\wedge,\vee,\hat 0)$ be a poset, $f$ a semimultiplicative function on $P$ such that $f(x)\neq 0$ for all $x\in P$, and let $S=\{x_1,x_2,\ldots,x_n\}$ be a finite, lower closed subset of $P$ such that $x_i\preceq x_j$ only if $i\leq j$. Define $g(x)=1/f(x)$ for $x\in P$ and let
\[
\sum_{\substack{z\preceq x_i\\ z\not\preceq x_j,\ j<i}}(g_d*\mu)(\hat 0,z)\neq 0\quad\text{for all }x\in\{1,2,\ldots,n\}.
\]

(i) Each $[S]_f$-eigenvalue $\lambda\in\mathbb{C}$ of $(S)_f$ lies in one of the disks
\[
\left\{z\in\mathbb{C}:\bigg|z-l_i^{-1}\sum_{\alpha,\beta=1}^n\frac{\mu(x_\alpha,x_i)\mu(x_\beta,x_i)}{f(x_\alpha\vee x_\beta)}\bigg|\leq |l_i|^{-1}\sum_{\substack{1\leq j\leq n\\ j\neq i}}\bigg|\sum_{\alpha,\beta=1}^n\frac{\mu(x_\alpha,x_i)\mu(x_\beta,x_j)}{f(x_\alpha\vee x_\beta)}\bigg|\right\},\quad i\in\{1,2,\ldots,n\}.
\]

(ii) The $[S]_f$-eigenvalues $\lambda\in\mathbb{C}$ of $(S)_f$ can be bounded from above by
\[
|\lambda|\leq c_f^{-1}m_g^{-1}F_{n+1}(F_{n+3}-2),
\]
where
\[
c_f=\min_{1\leq i,j\leq n}|f(x_i\vee x_j)|\ \text{and}\  m_g=\min_{1\leq i\leq n}\bigg|\sum_{\substack{z\preceq x_i\\ z\not\preceq x_j,\ j<i}}(g_d*\mu)(\hat 0,z)\bigg|
\]
and $(F_i)_{i=1}^\infty$ denotes the Fibonacci sequence defined by the recursion $F_1=F_2=1$ and $F_{i+2}=F_i+F_{i+1}$ for $i\in\mathbb{Z_+}$.
\end{theorem}
\begin{proof}
The first claim of the theorem follows immediately from the Gerschgorin theorem (see, for example, \cite[Theorem~7.2.1]{GolubVanLoan}) by applying Lemma~\ref{gcdlcmelem}. To obtain the upper bound for the $[S]_f$-eigenvalues $\lambda\in\mathbb{C}$ of $(S)_f$, we note that there exists some $k\in\{1,2,\ldots,n\}$ such that we can estimate
\[
|\lambda|\leq \sum_{j=1}^n\left|l_k^{-1}\sum_{\alpha=1}^nE_{k,\alpha}^{-1}\sum_{\gamma=1}^n\frac{f(x_\alpha\wedge x_\gamma)E_{j,\gamma}^{-1}}{f(x_\alpha)f(x_\gamma)}\right|\leq c_f^{-1}m_g^{-1}\sum_{\alpha=1}^n|E_{k,\alpha}^{-1}|\sum_{j=1}^n\sum_{\gamma=1}^n|E_{j,\gamma}^{-1}|,
\]
where
\[
c_f=\min_{1\leq i,j\leq n}|f(x_i\vee x_j)|\ \text{and}\  m_g=\min_{1\leq i\leq n}\bigg|\sum_{\substack{z\preceq x_i\\ z\preceq x_j,\ j<i}}(g_d*\mu)(\hat 0,z)\bigg|.
\]

It has been shown by Alt{\i}n{\i}\c{s}{\i}k et al.~\cite{ihk} that if $X\in K(n)$, then the diagonal elements of $X^{-1}=(Y_{i,j})_{i,j=1}^n$ are $Y_{i,i}=1$ for $i\in\{1,2,\ldots,n\}$ and the off-diagonal entries can be estimated by
\[
|Y_{i,j}|\leq F_{i-j}\quad\text{for }1\leq j<i\leq n\ \text{and}\  Y_{i,j}=0\quad\text{otherwise},
\]
where $(F_i)_{i=1}^\infty$ is the Fibonacci sequence defined by the recursion $F_1=F_2=1$ and $F_{i+2}=F_{i}+F_{i+1}$ for $i\in\mathbb{Z_+}$. In addition, the partial sums of the Fibonacci sequence satisfy the easily verified identity
\[
\sum_{i=1}^nF_i=F_{n+2}-1\quad\text{for }n\in\mathbb{Z_+}.
\]

The previous discussion allows us to carry out the estimations
\[
\sum_{\alpha=1}^n|E_{k,\alpha}^{-1}|\leq 1+\sum_{\alpha=1}^{k-1} F_{k-\alpha}=F_{k+1}\leq F_{n+1}\quad\text{for all }k\in\{1,2,\ldots,n\}
\]
and
\[
\sum_{j=1}^n\sum_{\gamma=1}^n|E_{j,\gamma}^{-1}|\leq\sum_{j=1}^nF_{j+1}=F_{n+3}-2.
\]
Hence
\[
|\lambda|\leq c_f^{-1}m_g^{-1}F_{n+1}(F_{n+3}-2),
\]
which is the desired result.
\end{proof}

In the course of the proof of part (ii) of Theorem~\ref{localbound1}, we have estimated the elements of $E^{-1}$, i.e., the values of the M\"{o}bius function on the subset $S$ of lattice $P$, using a somewhat crude argumentation pertaining to Fibonacci numbers. However, in semilattices where the behavior of the M\"{o}bius function is known a priori, some additional steps can be taken to further improve the bounds. We demonstrate one such method in the following section with respect to the divisor lattice.

\subsection{Eigenvalues of GCD matrices with respect to LCM matrices}
When the properties of the lattice are known, the methodology presented in the previous section permits for the derivation of even tighter bounds. In this section, we consider as an example the eigenvalues of GCD matrices with respect to LCM matrices in the lattice $(P,\preceq)=(\mathbb{Z_+},|)$.

Let $S=\{x_1,x_2,\ldots,x_n\}\subset\mathbb{Z_+}$ be a finite, factor closed set ordered such that $x_i\leq x_j$ whenever $i\leq j$. Let $f$ be a complex-valued, semimultiplicative function defined on $\mathbb{Z_+}$ such that $f(x)\neq 0$ for all $x\in\mathbb{Z_+}$. Let $(S)_f$ denote the GCD matrix on $S$,
\[
((S)_f)_{i,j}=f((x_i,x_j))
\]
and let $[S]_f$ denote the LCM matrix on $S$,
\[
([S]_f)_{i,j}=f([x_i,x_j]).
\]
Let $g(x)=1/f(x)$ for $x\in\mathbb{Z_+}$. We require that the LCM matrix $[S]_f$ is invertible, which is equivalent to the condition
\[
l_i=\sum_{\substack{z|x_i\\ z\not |x_j,\ j<i}}(g*_D\mu)(z)\neq 0,\quad i\in\{1,2,\ldots,n\}.
\]

Let $E$ denote the $n\times n$ matrix defined element-wise by setting
\[
E_{i,j}=\begin{cases}1&\text{if }x_j|x_i,\\ 0&\text{otherwise.}\end{cases}
\]
Let $R={\rm diag}(f(x_1),\ldots,f(x_n))$, $L={\rm diag}(l_1,\ldots,l_n)$, and let $D={\rm diag}(d_1,\ldots,d_n)$ with
\[
d_i=\sum_{\substack{z|x_i\\ z\not |x_j,\ j<i}}(f*_D\mu)(z),\quad i\in\{1,2,\ldots,n\}.
\]
It now follows from Propositions~\ref{EDET} and~\ref{RELETR} that
\[
(S)_f=EDE^\textup{T}\ \text{and}\ [S]_f=RELE^\textup{T}R.
\]

Part (i) of Theorem~\ref{localbound1} implies that
\[
|\lambda|\leq m_g^{-1}c_f^{-1}\sum_{\alpha=1}^n|\mu(x_\alpha,x_i)|\sum_{j=1}^n\sum_{\beta=1}^n|\mu(x_\beta,x_j)|,
\]
where
\[
c_f=\min_{1\leq i,j\leq n}|f([x_i,x_j])|\ \text{and}\ m_g=\min_{1\leq i\leq n}\bigg|\sum_{\substack{z|x_i\\ z\not |x_j,\ j<i}}(g*_D\mu)(z)\bigg|
\]
and $\mu(x,y)=\mu(y/x)$ is the arithmetical M\"{o}bius function for $x|y$, $y\in\mathbb{Z_+}$.

Let $k\in\mathbb{Z_+}$ have the prime decomposition $k=p_1^{a_1}\cdots p_r^{a_r}$. We define the arithmetical functions $\omega(k)=r$ and $\Omega(k)=a_1+\ldots+a_r$. Moreover, let $\tau(k)=\sum_{d|k}1$ be the number-of-divisors function. These functions have the following easily verified properties (for more on these arithmetical functions, see for example~\cite[Section~22.13]{HW})
\[
2^{\omega(k)}\leq \tau(k)\leq 2^{\Omega(k)}\quad\text{for all }k\in \mathbb{Z_+}
\]
and
\[
2^{\omega(k)}=\tau(k)=2^{\Omega(k)}\quad\text{for any squarefree }k\in\mathbb{Z_+}.
\]

We observe that $\mu(d)\neq 0$ for $d|k$ if and only if $|\mu(d)|=1$ if and only if $d=\prod_{\alpha\in I}\alpha$, where $I\subseteq\{p_1,\ldots,p_{\omega(k)}\}$, i.e., $I$ is a subset of the distinct prime factors of $k$. Since $S$ is factor closed, the sum can therefore be related to the cardinality of this power set via
\[
\sum_{\beta=1}^n|\mu(x_\beta,x_j)|\leq\#\{I\mid I\subseteq \{p_1,\ldots,p_{\omega(x_j)}\}\}=2^{\omega(x_j)}\quad\text{for all }j\in\{1,2,\ldots,n\}.
\]
In particular, it holds that
\begin{equation}
\sum_{\beta=1}^n|\mu(x_\beta,x_j)|\leq 2^{\omega(x_j)}\leq 2\sqrt{x_j}\quad \text{for all }j\in\{1,2,\ldots,n\}\label{naivebound}
\end{equation}
and
\[
\sum_{j=1}^n\sum_{\beta=1}^n|\mu(x_\beta,x_j)|\leq \sum_{j=1}^n\tau(x_j)\leq x_n\log x_n+(2\gamma-1)x_n+\mathcal{O}(\sqrt{x_n}),
\]
where $\gamma$ denotes the Euler--Mascheroni constant and the latter inequality follows from the average order of the number-of-divisors function~\cite[Theorem~320]{HW}. However, the bound~\eqref{naivebound} can be improved for $x_n>70$ by replacing the bound~\eqref{naivebound} with the following result of Robin~\cite{Robin}
\[
\omega(k)\leq 1.3841\displaystyle\frac{\log k}{\log\log k}\quad\text{for }k>2,
\]
implying that
\[
\sum_{\beta=1}^n|\mu(x_\beta,x_j)|\leq 2^{\omega(x_j)}\leq x_n^{0.9594/\log\log x_n}\leq 2\sqrt{x_n}\quad\text{for }x_n>70.
\]

The previous discussion yields the following corollary to the local bound of the previous section.
\begin{corollary}
Let $(P,\preceq,\wedge,\vee)=(\mathbb{Z_+},|,{\rm gcd},{\rm lcm})$, $f$ a complex-valued semimultiplicative function on $\mathbb{Z_+}$ such that $f(x)\neq 0$ for all $x\in\mathbb{Z_+}$, and $S=\{x_1,x_2,\ldots,x_n\}\subset\mathbb{Z_+}$ a finite, factor closed set ordered such that $x_i\leq x_j$ whenever $i\leq j$. Let $g(x)=1/f(x)$ for $x\in\mathbb{Z_+}$ and suppose that
\[
\sum_{\substack{z|x_i\\ z\not |x_j,\ j<i}}(g*_D\mu)(x)\neq 0\quad\text{for all }i\in\{1,2,\ldots,n\}.
\]
Then the $[S]_f$-eigenvalues $\lambda\in\mathbb{C}$ solving
\[
(S)_fx=\lambda[S]_fx
\]
for some $x\in\mathbb{C}^n\setminus\{0\}$ satisfy
\[
|\lambda|\leq 2c_f^{-1}m_g^{-1}(x_n^{3/2}\log x_n+(2\gamma-1)x_n^{3/2}+\mathcal{O}(x_n))\quad\text{for all }x_n\in\mathbb{Z_+},
\]
where $\gamma$ denotes the Euler--Mascheroni constant and
\[
c_f=\min_{1\leq i,j\leq n}|f([x_i,x_j])|\ \text{and}\ m_g=\min_{1\leq i\leq n}\bigg|\sum_{\substack{z|x_i\\ z\not |x_j,\ j<i}}(g*_D\mu)(z)\bigg|.
\]
For $x_n>70$, the upper bound can be improved by
\[
|\lambda|\leq c_f^{-1}m_g^{-1}x_n^{0.9594/\log\log x_n}(x_n\log x_n+(2\gamma-1)x_n+\mathcal{O}(\sqrt{x_n})).
\]
\end{corollary}

Remarkably, the upper bound can be bounded by a polynomial term corrected by a moderate logarithmic term. This is an improvement over the local bound of Theorem~\ref{localbound1}, which bounds the $[S]_f$-eigenvalues of $(S)_f$ by an exponential term that arises from the product of two Fibonacci numbers.

\section{Eigenvalues of join matrices with respect to meet matrices}\label{lcmgcdsec}
We consider the dual eigenvalue problem of finding $(\lambda,x)\in\mathbb{C}\times(\mathbb{C}^n\setminus\{0\})$ such that
\[
[S]_fx=\lambda(S)_fx,
\]
where $\lambda$ is called the $(S)_f$-eigenvalue and $x$ the $(S)_f$-eigenvector of $[S]_f$.

The proofs of the following results are obtained analogously to those in Section~\ref{gcdlcmsec} and are thus omitted.
\begin{lemma}\label{lcmgcdlemma}
Let $(P,\preceq,\wedge,\vee,\hat 0)$ be a poset and $f$ a complex-valued, semimultiplicative function defined on $P$ such that $f(x)\neq 0$ for all $x\in P$ and define $g(x)=1/f(x)$. Let $S=\{x_1,x_2,\ldots,x_n\}$ be a finite, meet closed subset of $P$ such that $x_i\preceq x_j$ only if $i\leq j$. Let
\[
d_i=\sum_{\substack{z\preceq x_i\\ z\not\preceq x_j,\ j<i}}(f_d*\mu)(\hat 0,z)\neq 0\quad\text{for all }i\in\{1,2,\ldots,n\}.
\]
Then the $(S)_f$-eigenvalues $\lambda$ solving
\[
[S]_fx=\lambda(S)_fx
\]
for some $x\in\mathbb{C}^n\setminus\{0\}$ are precisely the eigenvalues of the system
\[
D^{-1}By=\lambda y,
\]
where $y\in\mathbb{C}^n\setminus\{0\}$, $D={\rm diag}(d_1,\ldots,d_n)$, and $B=E^{-1}[S]_f(E^{-1})^\textup{T}$. The $n\times n$ matrix $E$ is defined as in~\eqref{Emat}.
\end{lemma}
\begin{corollary}
In addition to the assumptions of Lemma~\ref{lcmgcdlemma}, let
\[
l_i=\sum_{\substack{z\preceq x_i\\ z\not\preceq x_j,\ j<i}}(g_d*\mu)(\hat 0,z)\geq 0\quad\text{for all }i\in\{1,2,\ldots,n\}.
\]
Then the $(S)_f$-eigenvalues of $[S]_f$ are real and they can be ordered $(\lambda_i)_{i=1}^n$ (counting multiplicities) such that
\[
{\rm sign}\,\lambda_i={\rm sign}\sum_{\substack{z\preceq x_i\\ z\not\preceq x_j,\ j<i}}(f_d*\mu)(\hat 0,z)\quad\text{for all }i\in\{1,2,\ldots,n\}.
\]
\end{corollary}
\subsection{Global bounds for the $(S)_f$-eigenvalues of $[S]_f$}
\begin{theorem}
Let $(P,\preceq,\wedge,\vee,\hat 0)$ be a poset, $f$ a semimultiplicative function on $P$ such that $f(x)\neq 0$ for all $x\in P$, and let $S=\{x_1,x_2,\ldots,x_n\}$ be a finite, meet closed subset of $P$ such that $x_i\preceq x_j$ only if $i\leq j$. Define $g(x)=1/f(x)$ for $x\in P$ and let
\[
\sum_{\substack{z\preceq x_i\\ z\not\preceq x_j,\ j<i}}(f_d*\mu)(\hat 0,z)\neq 0\quad\text{for all }i\in\{1,2,\ldots,n\}.
\]
Then the $(S)_f$-eigenvalues $\lambda\in\mathbb{C}$ of $[S]_f$ are bounded from above by
\[
|\lambda|\leq\max_{1\leq i\leq n}|f(x_i)|^2\cdot M_gm_f^{-1}C_nc_n^{-1},
\]
where
\[
M_g=\max_{1\leq i\leq n}\bigg|\sum_{\substack{z\preceq x_i\\ z\not\preceq x_j,\ j<i}}(g_d*\mu)(\hat 0,z)\bigg|\ \text{and}\  m_f=\min_{1\leq i\leq n}\bigg|\sum_{\substack{z\preceq x_i\\ z\not\preceq x_j,\ j<i}}(f_d*\mu)(\hat 0,z)\bigg|.
\]
\end{theorem}
If in addition it holds that
\[
\sum_{\substack{z\preceq x_i\\ z\not\preceq x_j,\ j<i}}(g_d*\mu)(\hat 0,z)\neq 0\quad\text{for all }i\in\{1,2,\ldots,n\},
\]
then the $(S)_f$-eigenvalues $\lambda\in\mathbb{C}\setminus\{0\}$ of $[S]_f$ are bounded from below by
\[
|\lambda|\geq\min_{1\leq i\leq n}|f(x_i)|^2\cdot m_gM_f^{-1}C_n^{-1}c_n,
\]
where
\[
m_g=\min_{1\leq i\leq n}\bigg|\sum_{\substack{z\preceq x_i\\ z\not\preceq x_j,\ j<i}}(g_d*\mu)(\hat 0,z)\bigg|\ \text{and}\  M_f=\max_{1\leq i\leq n}\bigg|\sum_{\substack{z\preceq x_i\\ z\not\preceq x_j,\ j<i}}(f_d*\mu)(\hat 0,z)\bigg|.
\]
\subsection{Local bounds for the $(S)_f$-eigenvalues of $[S]_f$}
\begin{lemma}
In addition to the assumptions of Lemma~\ref{lcmgcdlemma}, let $S$ be lower closed. Then it holds that
\[
(D^{-1}B)_{i,j}=d_i^{-1}\sum_{\alpha,\beta=1}^n\mu(x_\alpha,x_i)\mu(x_\beta,x_j)f(x_\alpha\vee x_\beta)\quad\text{for }i,j\in\{1,2,\ldots,n\}.
\]
\end{lemma}
\begin{theorem}
Let $(P,\preceq,\wedge,\vee,\hat 0)$ be a poset, $f$ a semimultiplicative function such that $f(x)\neq 0$ for all $x\in P$, and let $S=\{x_1,x_2,\ldots,x_n\}$ be a finite, lower closed subset of $P$ such that $x_j\preceq x_i$ only if $i\leq j$. Define $g(x)=1/f(x)$ for $x\in P$ and let
\[
\sum_{\substack{z\preceq x_i\\ z\not\preceq x_j,\ j<i}}(f_d*\mu)(\hat 0,z)\neq 0\quad\text{for all }x\in\{1,2,\ldots,n\}.
\]

(i) Each $(S)_f$-eigenvalue $\lambda\in\mathbb{C}$ of $[S]_f$ lies in one of the disks
\[
\bigg\{z\in\mathbb{C}: \bigg|z-d_i^{-1}\sum_{\alpha,\beta=1}^n\mu(x_\alpha,x_i)\mu(x_\beta,x_i)f(x_\alpha\vee x_\beta)\bigg|\leq |d_i|^{-1}\sum_{\substack{1\leq j\leq n\\ j\neq i}}\bigg|\sum_{\alpha,\beta=1}^n\mu(x_\alpha,x_i)\mu(x_\beta,x_j)f(x_\alpha\vee x_\beta)\bigg|\bigg\},
\]
$i\in\{1,2,\ldots,n\}$.

(ii) The $(S)_f$-eigenvalues $\lambda\in\mathbb{C}$ of $[S]_f$ can be bounded from above by
\[
|\lambda|\leq C_fm_f^{-1}F_{n+1}(F_{n+3}-2),
\]
where
\[
C_f=\max_{1\leq i,j\leq n}|f(x_i\vee x_j)|\ \text{and}\  m_f=\min_{1\leq i\leq n}\bigg|\sum_{\substack{z\preceq x_i\\ z\preceq x_j,\ j<i}}(f_d*\mu)(\hat 0,z)\bigg|
\]
and $(F_i)_{i=1}^\infty$ denotes the Fibonacci sequence defined by the recursion $F_1=F_2=1$ and $F_{i+2}=F_i+F_{i+1}$ for $i\in\mathbb{Z_+}$.
\end{theorem}
\section{Numerical experiments}\label{experiments}
In this section, we consider the generalized spectra for two selected geometrical lattices and compare their generalized eigenvalues to the local bounds obtained in Section~\ref{gcdlcmsec}.
\begin{example}
We consider the $[S]_f$-eigenvalues of $(S)_f$, when $S=\{1,2,...,8\}$ with the induced partial ordering $\preceq$ and embedded function $f$ determined by the Hasse diagram on the left-hand side of Figure~\ref{hasse1}. The matrices $(S)_f=EDE^\textup{T}$ and $[S]_f=RELE^\textup{T}R$ can now be constructed by their factors: the induced partial order $\preceq$ produces the matrix
\[
E=\left[\begin{array}{cccccccc}
1&0&0&0&0&0&0&0\\ 
1&1&0&0&0&0&0&0\\ 
1&0&1&0&0&0&0&0\\ 
1&1&1&1&0&0&0&0\\ 
1&1&1&1&1&0&0&0\\ 
1&1&1&1&1&1&0&0\\  
1&1&1&1&1&0&1&0\\ 
1&1&1&1&1&1&1&1
\end{array}\right]
\]
and the endowed restricted incidence function $f$ yields
\[
R={\rm diag}(1,2,3,6,4,5,8,10),\ D={\rm diag}(1,1,2,2,-2,1,4,1),\ \text{and}\ L={\rm diag}\left(1,-\frac{1}{2},-\frac{2}{3},\frac{1}{3},\frac{1}{12},-\frac{1}{20},-\frac{1}{8},\frac{1}{40}\right).
\] 
Theorem~\ref{localbound1} implies that the Gerschgorin disks have the midpoints
\[
(w_1,w_2,w_3,w_4,w_5,w_6,w_7,w_8)=(0,0,0,-1,-1,0,0,0)
\] 
and the radii
\[
(r_1,r_2,r_3,r_4,r_5,r_6,r_7,r_8)=\left(5,\frac{1}{2},\frac{1}{3},\frac{5}{6},2,\frac{4}{5},\frac{1}{2},\frac{8}{5}\right).
\]

The eigenvalues and associated Gerschgorin disks $\{z\in\mathbb{C}:|z-w_i|\leq r_i\}$, $i\in\{1,2,\ldots,8\}$, of Theorem~\ref{localbound1} are displayed on the right-hand side of Figure~\ref{hasse1}. We observe that the Gerschgorin disks are clustered near the origin with reasonably scaled radii.
\begin{figure}[!h]
\centering
\includegraphics{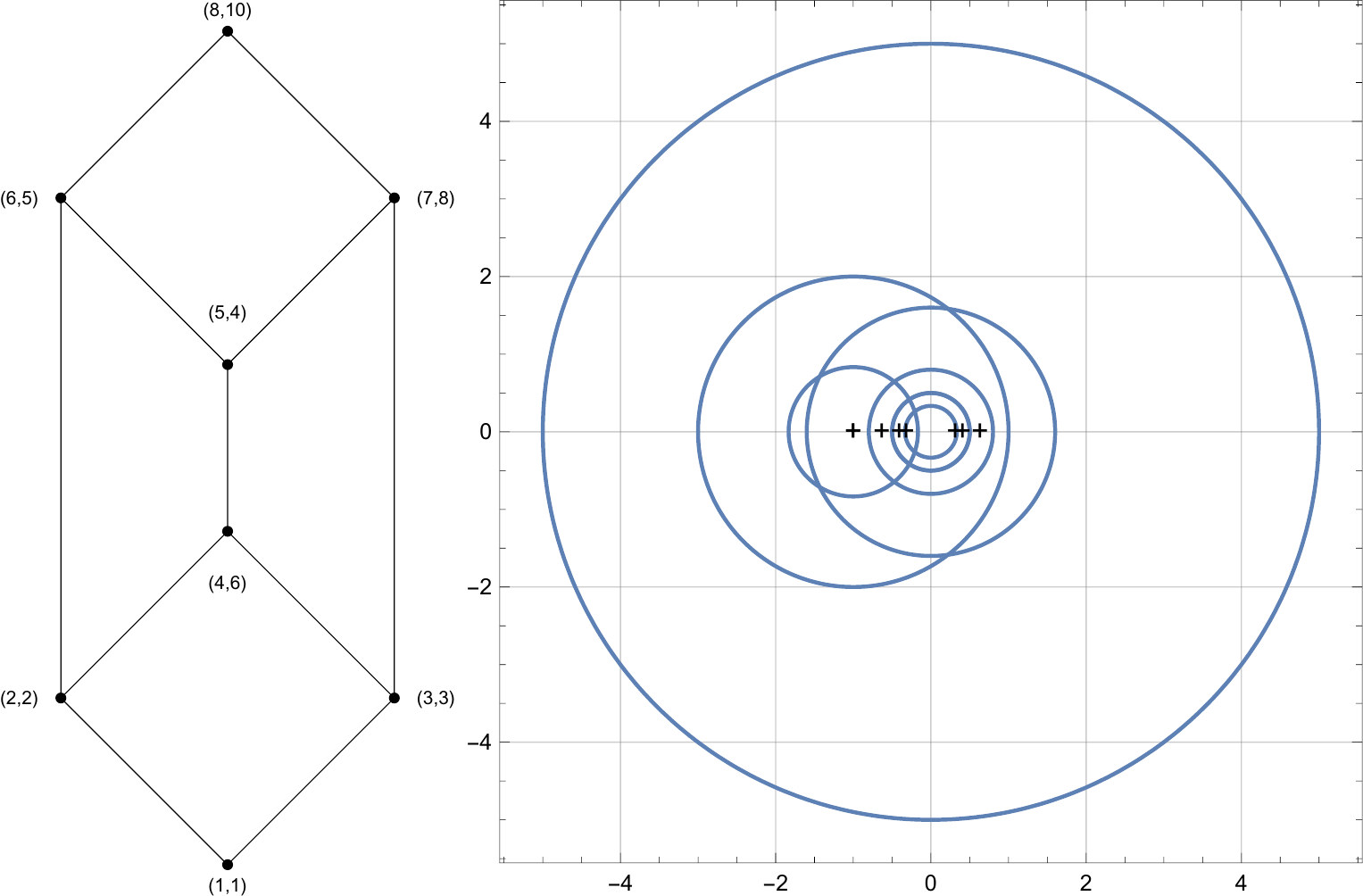}
\caption{Left: Hasse diagram of pairs $(x,f(x))\in S\times f(S)$ and the induced partial ordering $\preceq$, where elements that are connected by a line segment are ordered by $\prec$ from bottom to top. Right: The Gerschgorin disks of Theorem~\ref{localbound1} and the computed $[S]_f$-eigenvalues of $(S)_f$.}\label{hasse1}
\end{figure}
\end{example}

\begin{example}
We consider the $[S]_f$-eigenvalues of $(S)_f$, when $S=\{1,2,\ldots,12\}$ with the induced partial ordering $\preceq$ and embedded function $f$ determined by the Hasse diagram on the left-hand side of Figure~\ref{hasse2}. The matrices $(S)_f=EDE^\textup{T}$ and $[S]_f=RELE^\textup{T}R$ can be constructed factor-wise: the partial ordering $\preceq$ now produces the matrix
\[
E=\left[\begin{array}{cccccccccccc}
1&0&0&0&0&0&0&0&0&0&0&0\\
1&1&0&0&0&0&0&0&0&0&0&0\\
1&0&1&0&0&0&0&0&0&0&0&0\\
1&1&1&1&0&0&0&0&0&0&0&0\\
1&1&1&1&1&0&0&0&0&0&0&0\\
1&1&1&1&0&1&0&0&0&0&0&0\\
1&1&1&1&1&1&1&0&0&0&0&0\\
1&1&1&1&1&1&0&1&0&0&0&0\\
1&1&1&1&1&1&1&1&1&0&0&0\\
1&1&1&1&1&1&1&1&1&1&0&0\\
1&1&1&1&1&1&1&1&1&0&1&0\\
1&1&1&1&1&1&1&1&1&1&1&1
\end{array}\right]
\]
and the endowed restricted incidence function $f$ yields
\begin{align*}
&R={\rm diag}(1,2,4,8,4,6,6,4,3,5,9,45),\ D={\rm diag}(1,1,3,3,-4,-2,4,2,-5,2,6,34),\\
\text{and}\ &L={\rm diag}\left(1,-\frac{1}{2},-\frac{3}{4},\frac{3}{8},\frac{1}{8},\frac{1}{24},-\frac{1}{8},-\frac{1}{24},\frac{5}{24},-\frac{2}{15},-\frac{2}{9},\frac{2}{45}\right)
\end{align*}
Theorem~\ref{localbound1} implies that the Gerschgorin disks have the midpoints
\[
(w_1,w_2,w_3,w_4,w_5,w_6,w_7,w_8,w_9,w_{10},w_{11},w_{12})=\left(0,0,0,-1,-\frac{32}{15},-\frac{32}{15},-\frac{6}{5},-\frac{6}{5},-1,1,1,0\right)
\] 
and the radii
\[
(r_1,r_2,r_3,r_4,r_5,r_6,r_7,r_8,r_9,r_{10},r_{11},r_{12})=\left(9,\frac{1}{2},\frac{1}{4},
\frac{25}{12},\frac{38}{15},\frac{44}{45},\frac{14}{45},\frac{16}{5},8,\frac{12}{5},\frac{8}{9},\frac{52}{45}\right).
\]

The eigenvalues and associated Gerschgorin disks $\{z\in\mathbb{C}:|z-w_i|\leq r_i\}$, $i\in\{1,2,\ldots,12\}$, of Theorem~\ref{localbound1} are displayed on the right-hand side of Figure~\ref{hasse2}. The Gerschgorin disks are clustered near the origin with reasonably small radii.

\begin{figure}[!h]
\centering
\includegraphics{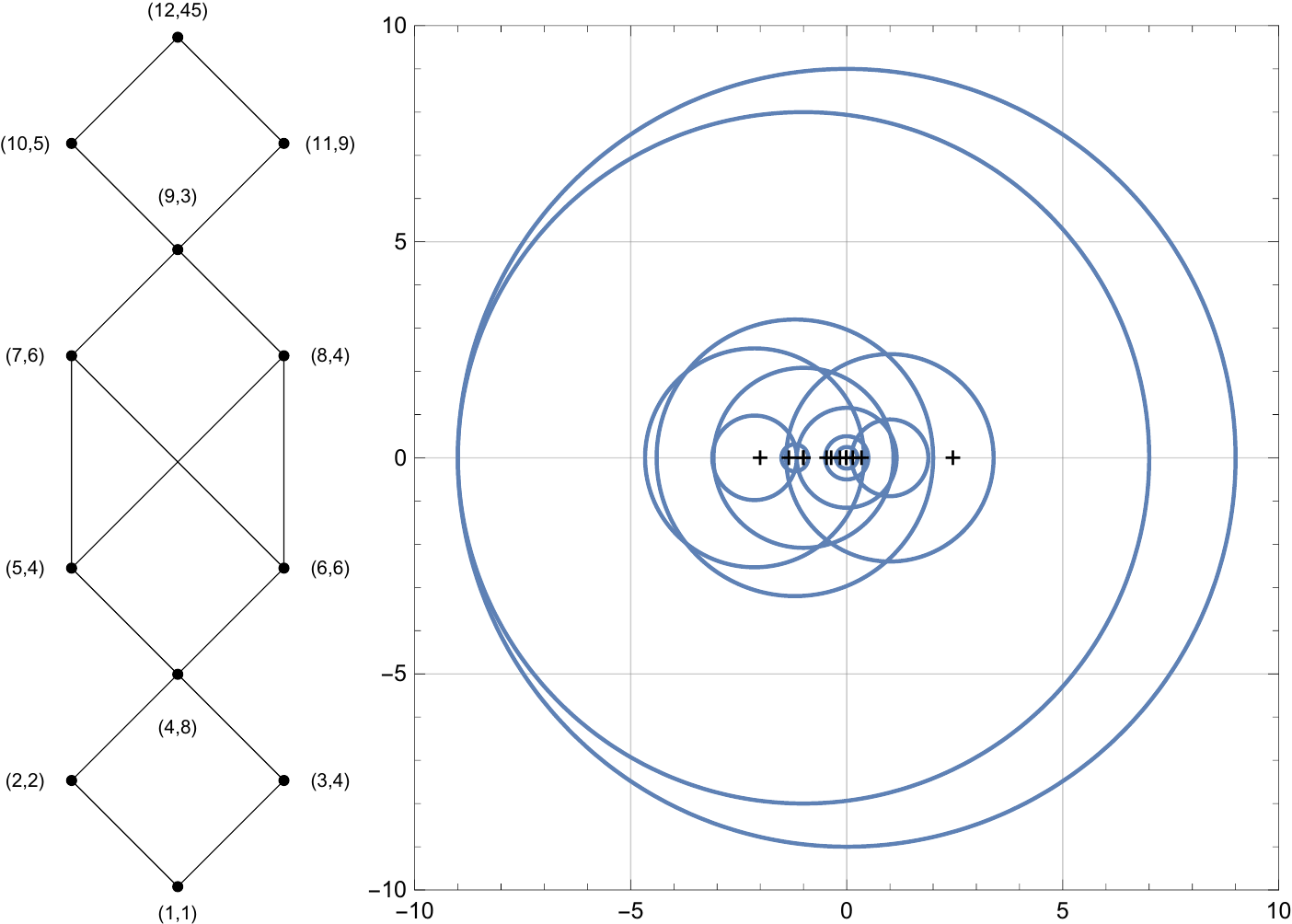}
\caption{Left: Hasse diagram of pairs $(x,f(x))\in S\times f(S)$ and the induced partial ordering $\preceq$, where elements that are connected by a line segment are ordered $\prec$ from bottom to top. Right: The Gerschgorin disks of Theorem~\ref{localbound1} and the computed $[S]_f$-eigenvalues of $(S)_f$.}\label{hasse2}
\end{figure}
\end{example}
\section{Conclusions and future prospects}\label{conclusions}
We have conducted the first study of the properties of the generalized eigenvalues of meet and join matrices on meet closed semilattices. Novel bounds for the smallest and largest eigenvalues were derived in several distinct cases: we obtained global bounds that uniformly govern the behavior of the generalized eigenvalues of a large class of meet and join matrices. We have also proposed new local bounds, which utilize information that can be discerned directly from a given lattice's properties to obtain improved bounds. We have demonstrated that our results are applicable in both the case of eigenvalues of meet matrices with respect to join matrices and vice versa. As a case study for demonstrating how the usage of local eigenvalue bounds improves the global eigenvalue bounds, we have considered the select example of the eigenvalues of GCD matrices with respect to LCM matrices. The dominant eigenvalues in this case can be shown to increase at a polynomial rate, which is an immediate improvement over the global bound that is known to increase exponentially. Also, two geometrical lattices were selected to demonstrate the usage of the local bounds numerically and the eigenvalues were found to agree well with the theoretical bounds proposed in this work.

Semilattices in which the behavior of the associated M\"{o}bius function is known a priori appear to be most inviting in view of our proposed approach to obtain sharper bounds on generalized eigenvalues --- such as the Boolean algebra or the unitary divisor lattice embedded with the greatest common unitary divisor and the least common unitary multiple. We are confident that the methods proposed in this paper can be utilized successfully in these aforementioned cases as well as other similar problems.

\section*{Acknowledgement}

The authors wish to thank the anonymous referee for his/her insightful comments and suggestions which helped to improve this paper greatly.

\section*{References}

\bibliography{IK}

\end{document}